\documentclass{amsart}

\usepackage{amsmath,
tikz,circuitikz,graphicx,enumerate,xcolor,amssymb,
mathtools,float,soul,booktabs,comment,appendix,
lmodern,enumitem,todonotes,hyperref,multirow}
\newtheorem{theorem}{Theorem}[section]
\newtheorem{lem}[theorem]{Lemma}

\theoremstyle{definition}
\newtheorem{defn}[theorem]{Definition}
\newtheorem{example}[theorem]{Example}
\newtheorem{prop}[theorem]{Proposition}
\newtheorem{nota}[theorem]{Notation}

\theoremstyle{remark}
\newtheorem{rem}[theorem]{Remark}

\numberwithin{equation}{section}

\newcommand{\C}{\mathbb{C}}
\newcommand{\B}{\mathbb{B}}

\newcommand{\I}{\mathbb{I}}

\DeclareMathOperator{\End}{End}

\DeclareMathOperator{\id}{id}

\newcommand{\pv}{\textcolor{blue}}

\begin{document}

\title[Gaussian solutions and their twists]{Gaussian solutions to the Yang--Baxter equation and their twists}

\author{Yasmeen S. Baki}

\address{Arizona State University}
\email{ybaki1@asu.edu}
\thanks{The first author was supported, in part, by Indiana University Northwest's 2025-2026 Grant-In-Aid of Research.}

\author{Padmini Veerapen}
\address{Indiana University Northwest}
\email{pveerap@iu.edu}
\thanks{The second author was supported, in part, by Indiana University Northwest's 2026 Summer Faculty Fellowship for Research.}

\subjclass[2020]{Primary 16T10, 16T25; Secondary 16S37, 16W20}
\date{July 1, 2026 and, in revised form, June 22, 1994.}

\dedicatory{The second author dedicates this paper to James J. Zhang's continuous support of noncommutative ring theorists.}

\keywords{bialgebras, quadratic algebras, Yang--Baxter equations, automorphisms}

\begin{abstract}
In this paper, we consider two explicit Gaussian solutions to the constant (or  parameter-independent) quantum Yang--Baxter equation and compute the corresponding bialgebras using the Faddeev--Reshetikhin--Takhtajan construction (FRT). Additionally, we twist these two Gaussian solutions, via Zhang twists and corresponding 2-cocycle twists, to obtain solutions to the Yang--Baxter equation which are not necessarily Gaussian.
\end{abstract}

\maketitle

\section{Introduction}
The Yang--Baxter (Y--B) equation is famous for its far-reaching applications in areas as diverse as representations of the braid group, invariant theory of knots and links, and quantum groups. Indeed, according to Jimbo \cite{Jimbo_Introduction}, work on the Y--B equation probably begun as early as 1944 in Onsager's paper on the Ising model before appearing in McGuire's in 1964 and Yang's in 1967 \cite{Onsager1944,McGuire1964,Yang1967}, and leading to Baxter's solution in 1972 in terms of a spectral parameter (known as parameter-dependent or parametric solution) \cite{Baxter1972,Baxter1982}. In the 1980s, the work on the Y--B equation gave rise to beautiful mathematics in the realms of representations of the braid group and quantum groups (see e.g., Figure \ref{table1} for a short summary).

\begin{figure}[!ht]
\begin{center}
\resizebox{.5\textwidth}{!}
{%
\begin{circuitikz}
\tikzstyle{every node}=[font=\fontsize{18.2pt}{23.7pt}\selectfont]
\draw  (-84.5,18.75) rectangle (-75.25,13.25);
\node[anchor=center, align=center, fill={rgb,255:red,255; green,255; blue,255}, fill opacity=1, text opacity=1, inner xsep=0.080cm, inner ysep=0.085cm, rounded corners=0.020cm] at (-79.875,16) {\fontsize{18.2pt}{23.7pt}\selectfont \textbf{Solution to} \\ \textbf{parameter-dependent} \\ \textbf{Yang-Baxter equation} \\ Baxter (1972) \\ \cite{Baxter1972,Baxter1982}};
\draw  (-69.875,17) rectangle (-58.625,11);
\node[anchor=center, align=center, fill={rgb,255:red,255; green,255; blue,255}, fill opacity=1, text opacity=1, inner xsep=0.080cm, inner ysep=0.085cm, rounded corners=0.020cm] at (-64.25,14) {\fontsize{18.2pt}{23.7pt}\selectfont \textbf{Quantum Groups} \\ Faddeev, Sklyanin, \& Takhtajan, \\ Kulish, Reshetikhin, \\ Drinfeld, Jimbo \\ (1979-1989) \\\cite{FST1979,Faddeev,KS11982,KS21982,Sklyanin,drinfeld,Jimbo1985}};
\draw [dashed] (-75.125,21.5) -- (-79.625,18.75);
\draw [dashed] (-71.75,21.375) -- (-65.25,17);
\draw  (-78.75,28.75) rectangle (-67.875,21.5);
\node[anchor=center, align=center, fill={rgb,255:red,255; green,255; blue,255}, fill opacity=1, text opacity=1, inner xsep=0.080cm, inner ysep=0.085cm, rounded corners=0.020cm] at (-73.3125,25.125) {\fontsize{18.2pt}{23.7pt}\selectfont \textbf{Yang-Baxter equation} \\ \textbf{(beginnings)} \\ \\ Onsager (1944) \cite{Onsager1944}, \\ McGuire (1964) \cite{McGuire1964}, \\ Yang (1967) \cite{Yang1967}};
\draw  (-86.375,8.125) rectangle (-72.25,0.75);
\node[anchor=center, align=center, fill={rgb,255:red,255; green,255; blue,255}, fill opacity=1, text opacity=1, inner xsep=0.080cm, inner ysep=0.085cm, rounded corners=0.020cm] at (-79.3125,4.4375) {\fontsize{18.2pt}{23.7pt}\selectfont \textbf{Representations of braid groups \&} \\ \textbf{solutions to Yang--Baxter equation} \\  \\ Jones, Goldschmidt, Rowell \\ (1989-2014) \\\cite{JonesBaxter1991,GJ1989,GalindoRowell2014}};
\draw [dashed] (-80.125,13.125) -- (-80.125,8.125);
\draw  (-87.875,30.125) rectangle (-57.125,-1.125);
\end{circuitikz}
}%
\caption{}
\end{center}
\label{table1}
\end{figure}

Over the many decades that followed the appearance of the Y--B equation, several families of solutions to the Y--B equation have been determined, ranging from \textit{classical} solutions to \textit{Gaussian} solutions. Hietarinta \cite{HIETARINTA1992245} classified solutions $R$ to the constant (or parameter-independent) quantum Y--B equation in dimension $n = 2$, where $R \in \End_{\Bbbk}(M) \otimes \End_{\Bbbk}(M)$, and $M$ is a 2-dimensional vector space. However, in dimension $n \ge 3$, the classification is incomplete. Our focus in this paper will be on \textit{Gaussian} solutions to the constant (or parameter-independent) quantum Y--B equation. We are motivated, in part, by Martin and Rowell's work on these solutions and their corresponding representations of the braid group, $\B_n$ \cite{MR2023}. As it were, the association between solutions to the Y--B equation and Artin's presentation of the $n$-string braid group was made initially by Jones \cite{Jones1985}, and continued in the work of Jones and Goldschmidt in the 1980s, where they considered Gaussian representations of $\B_n$ in the case where $n$ is a positive odd integer \cite{GJ1989,Jonespoly1989,Jonesstats1989}. This was later extended to the case where $n$ is a positive even integer by Galindo and Rowell in \cite{GalindoRowell2014}.

We have two main goals in this paper. The first is to further examine quantum groups or bi/Hopf algebras from the perspective of Gaussian solutions by applying the work of Faddeev, Reshetikhin, and Takhtajan via the FRT construction \cite{FRT} to produce coquasitriangular bialgebras. The second goal is to twist Gaussian solutions corresponding to the constant quantum Y--B equation when $n = 2$ via the mechanism of twisting pairs (i.e., pairs of algebra automorphisms of a bialgebra) defined in work by the second author \cite[Definition D]{HNUVVW1}.

With the widespread interest in quantum computing, the work of Kauffman and Lomonaco \cite{Kauffman_Lomonaco_2004} on unitary solutions to the Y--B equation is pertinent. Indeed, they show that unitary solutions to the Y--B equation are universal quantum gates. An example of such a solution is the Bell basis change matrix, $R_{+}$, given in \S \ref{sec:Gaussian}, Definition \ref{bellbasischangematrix_}. Rowell \cite{Kauffman_Lomonaco_2004,rowell2014} observes that $R_{+}$ is equivalent to a Gaussian solution in dimension $n = 2$, and that in higher dimensions the Gaussian solutions to the Y--B equation provide $n$-level generalizations of the Bell states.

In \S \ref{sec:Gaussian}, we produce an explicit list of Gaussian solutions, both unitary and non-unitary, in the case where $n = 2$, and in Proposition \ref{prop:equivclassslns} we partition the solutions into equivalence classes. In \S \ref{sec:FRTSect}, we use the FRT construction \cite{FRT} to compute the bialgebra $A(R)$, corresponding to the solution $R \in \End_{\Bbbk}(M) \otimes \End_{\Bbbk}(M)$, where $M$ is a 2-dimensional vector space. In Theorem \ref{thm:A(R1)_presentation}, we give its presentation and Appendix \ref{app:A1zijkl} lists the computations for the sixteen Y--B obstructions used to write $A(R)$ explicitly in terms of generators and relation. In \S \ref{sec:twists}, we examine the solutions to the Y--B equation obtained through twisting $R$ by a twisting pair, as described in \cite[\S 4.2]{HNUVVW1}. To do so, we classify all valid twisting pairs of $R$ in Theorem \ref{thm:valid_pairs} and Propositions \ref{prop:pairs_R1} and \ref{prop:pairs_Ri}, and proceed to compute all possible twists of $R$ in Theorems \ref{thm:twisted_R1} and \ref{thm:twisted_Ri}. As a result, we conclude that a twist of a Gaussian solution does not necessarily remain Gaussian. This suggests that finding additional solutions to the Y--B equation in dimensions $n \ge 3$ may be achieved by twisting higher-dimensional Gaussian solutions.     
    
\begin{nota}Throughout this paper, let the field $\Bbbk = \C$ and we write $\C^\times$ for the nonzero elements of $\C$. All vector spaces and all bialgebras are over the field $\Bbbk$. Unadorned tensors are understood to be over the base field $\Bbbk$. Any solution to the Y--B equation, $R_{\omega}$, for $\omega \in \mathbb{C}$, is scaled by $1/{\sqrt{2}}$. We use the Sweedler notation for the coproduct in a bialgebra $B$. That is, for any $h \in B$, $\Delta(h) = \sum h_1 \otimes h_2 \in B \otimes B$. The Einstein summation convention is being used. That is, if an index occurs twice in an expression once as an upper index and once as a lower index, then the sum is taken such that the index runs from $0$ to $1$. Moreover, in the notation $c^i_j$, $i$ is the column number and $j$ is the row number.
    \end{nota} 

    \subsection*{Acknowledgments} The first author is grateful to Indiana University Northwest, who sponsored a visit during which part of this paper was completed. The first author also acknowledges her use of ChatGPT Edu with GPT-5, with access provided by Arizona State University. In particular, GPT-5 was used in the preliminary stages of this project to discuss ideas, and later to help write Mathematica code which was used in the verification process. The entirety of this written work, along with its results, analysis, and conclusions are due to the first and second authors.   
    
    \section{The Yang--Baxter Equation and constant Gaussian Solutions}
    \label{sec:Gaussian}

    \subsection*{Constant versus parametric forms}
    The Y--B equation has many forms. One of its forms is the constant (also known as parameter-independent) form, and the other is the parametric (also known as parameter-dependent) form. Within each of these cases, we may further consider solutions to the \textit{quantum} Y--B equation and the \textit{braid equation} (BE). Our work will focus on solutions to the constant quantum Y--B equation, which is equivalent to solving the constant braid equation (cf. \cite[Proposition 2.2.1]{LambeRadford1997}). By ``solution to the Y--B equation'', we mean a solution to the \textit{constant} Y--B equation.
    
    \begin{defn}(Solution to the constant quantum Y--B equation \cite{CMZ,LambeRadford1997})
        Let $M$ be a $\Bbbk$-vector space, $R \in \End_{\Bbbk}(M) \otimes \End_{\Bbbk}(M)$ a linear map, and $\tau: M \otimes M \to M \otimes M$ the tensor flip map which sends $m \otimes n$ to $n \otimes m$, for $m, n \in M$. We provide the following definitions for $\{R^{ij}\}_{1 \leq i < j \leq 3} \in \End_{\Bbbk}(M) \otimes \End_{\Bbbk}(M) \otimes \End_{\Bbbk}(M)$
        \begin{align*}
            R^{12} &= R \otimes \id_M, \\
            R^{23} &= \id_M \otimes R, \\
            R^{13} &= (\id_M \otimes \tau)(R \otimes \id_M)(\id_M \otimes \tau).
        \end{align*}
        We say that $R$ is a solution to the constant quantum Y--B equation if the following holds
        \begin{align}
        \label{eq:cQYBE}
            R^{12}R^{13}R^{23} = R^{23}R^{13}R^{12}.
        \end{align}
This is equivalent to satisfying the braid equation
\begin{align*}
R^{12}R^{23}R^{12} = R^{23}R^{12}R^{23}.
\end{align*}
    \end{defn}

    \begin{defn}\cite[Definition 2.3.1]{LambeRadford1997}
    \label{def:equiv_sol}
        Let $M$ and $M'$ be two $k$-vector spaces and suppose $R \in  \End_k(M) \otimes \End_k(M)$ and $R' \in \End_k(M') \otimes \End_k(M')$ are two solutions to the Y--B equation. We consider $R$ and $R'$ to be equivalent if there exists an isomorphism $u: M \to M'$ such that 
        \begin{align*}
            R' = (u \otimes u) R (u^{-1} \otimes u^{-1}),
        \end{align*}    
    \end{defn}

    \subsection*{Gaussian Solutions}



The Bell basis change matrix has two forms $R_{-}$ and $R_{+}$. We provide both forms for completeness.
\begin{defn}\cite{Zhang_Jing_Ge_2008}
\begin{align}
\label{bellbasischangematrix_}
R_{-} = \frac{1}{\sqrt{2}}
\begin{pmatrix}
1 & 0 & 0 & 1 \\
0 & 1 & -1 & 0\\
0 & 1 & 1 &  0 \\
-1 & 0 & 0 & 1
\end{pmatrix}, 
&&  
R_{+} =
\frac{1}{\sqrt{2}}
\begin{pmatrix}
1 & 0 & 0 & 1 \\
0 & 1 & 1 & 0\\
0 & -1 & 1 &  0 \\
-1 & 0 & 0 & 1
\end{pmatrix}
\end{align}
\end{defn}

Kauffman and Lomonaco proved that $R_{+}$ satisfies the Y--B equation in \cite{Kauffman_Lomonaco_2004}. Additionally, Rowell notes that the Gaussian solutions to the Y--B equation are a family of transformations which generalize to higher-level Bell states, and which satisfy the usual topological braiding conditions \cite{rowell2014}. In particular, Gaussian solutions to the Y--B equation are given by 
\begin{align}
\label{eqtn:r_mat}
        R_{\omega} = \frac{1}{\sqrt{m}} \sum_{j=0}^{m-1} \omega^{j^2}U^j,
    \end{align}
    where $\omega$ is an $m^{th}$ root of unity when $m$ is odd, a $2m^{th}$ root of unity when $m$ is even, and $U \in \text{GL}(\mathbb{C}^{m^2})$ is such that $U(e_i \otimes e_j) = \omega^{i-j} e_{i-1} \otimes e_{j-1}$ where $\{e_i\}_{0 \leq i \leq m-1}$ is the standard basis for $\C^m$ (cf. \cite[\S 1]{rowell2014}).

    Note that in the case where $m = 2$ and $\omega = -i$, $R_{-}$ is equivalent to $R_{-i}$ in the sense of Definition \ref{def:equiv_sol} by taking 
    \begin{align*}
        u = 
        \begin{pmatrix}
            1 & 0 \\
            0 & e^{3 \pi i /4}
        \end{pmatrix}.
    \end{align*}
    
In the following example, we explicitly write out Gaussian solutions using Equation \ref{eqtn:r_mat} when $m=2$. Since it is necessary for $\omega$ to be a fourth root of unity, we have four options, yielding four $4 \times 4$ matrices. Let $\omega_1 = 1, \omega_2 = i, \omega_3 = -1$, and $\omega_4 = -i$, $\{e_0,e_1\}$ be the standard basis for $\mathbb{C}^2$, and 
    $\{e_{00},e_{01},e_{10}, e_{11}\}$ be an ordered basis for $\C^2 \otimes \C^2$. Note that when $j = 0$, $U^j = \I_4$, so $\omega_i^{j^2}U^j = \omega_i^0 U^0 = \I_4$ for all $1 \leq i \leq 4$.\\
    
\begin{example}
\label{egR1}
\leavevmode
\begin{enumerate}
    \item We compute the Gaussian solution, denoted by $R_1$, corresponding to $\omega_1 = 1$ and obtain,
    \begin{align}
    \label{sln1}
        R_1 &= 
        \frac{1}{\sqrt{2}}\left(
        \I_4 + 
        \begin{pmatrix}
        0 & 0 & 0 & 1 \\
        0 & 0 & 1 & 0 \\
        0 & 1 & 0 & 0 \\
        1 & 0 & 0 & 0
        \end{pmatrix}
        \right) =
        \frac{1}{\sqrt{2}}
        \begin{pmatrix}
            1 & 0 & 0 & 1 \\
            0 & 1 & 1 & 0 \\
            0 & 1 & 1 & 0 \\
            1 & 0 & 0 & 1
        \end{pmatrix}.
    \end{align}

    \item The Gaussian solution, $R_{-1}$, corresponding to $\omega_3 = -1$ is as follows.
    \begin{align}
        R_{-1} &= \frac{1}{\sqrt{2}}\left(
        \I_4 + 
        \begin{pmatrix}
            0 & 0 & 0 & -1 \\
            0 & 0 & 1 & 0 \\
            0 & 1 & 0 & 0 \\
            -1 & 0 & 0 & 0
        \end{pmatrix} 
        \right) = \frac{1}{\sqrt{2}}
        \begin{pmatrix}
            1 & 0 & 0 & -1 \\
            0 & 1 & 1 & 0 \\
            0 & 1 & 1 & 0 \\
            -1 & 0 & 0 & 1
        \end{pmatrix}.
    \end{align}
\item The Gaussian solution, $R_{i}$, corresponding to $\omega_2 = i$, is as follows.
    \begin{align}
    \label{eqn:Ri}
        R_{i} = \frac{1}{\sqrt{2}}\left(
        \I_4 +
        \begin{pmatrix}
            0 & 0 & 0 & i \\
            0 & 0 & -1 & 0 \\
            0 & 1 & 0 & 0 \\
            i & 0 & 0 & 0 
        \end{pmatrix}
        \right) = 
        \frac{1}{\sqrt{2}}
        \begin{pmatrix}
        1 & 0 & 0 & i \\
        0 & 1 & -1 & 0 \\
        0 & 1 & 1 & 0 \\
        i & 0 & 0 & 1
        \end{pmatrix}.
    \end{align}
\item The Gaussian solution, $R_{-i}$, corresponding to $\omega_4 = -i$, is as follows. 
    \begin{align}
    \label{R-i}
        R_{-i} &= \frac{1}{\sqrt{2}}\left(
        \I_4 + 
        \begin{pmatrix}
        0 & 0 & 0 & -i \\
        0 & 0 & -1 & 0 \\
        0 & 1 & 0 & 0 \\
        -i & 0 & 0 & 0 
        \end{pmatrix}
        \right) = \frac{1}{\sqrt{2}}
        \begin{pmatrix}
            1 & 0 & 0 & -i \\
            0 & 1 & -1 & 0 \\
            0 & 1 & 1 & 0 \\
            -i & 0 & 0 & 1
        \end{pmatrix}.
    \end{align}
\end{enumerate}
\end{example}

\begin{rem}
    The choices for roots of unity $\omega = 1, -1, i, -i$ result in solutions $R_1, R_{-1}, R_{i},$ and $R_{-i}$, as shown above. We observe that solutions $R_1$ and $R_{-1}$ yield non-unitary solutions while, $R_{i},$ and $R_{-i}$ yield unitary solutions. To go from a solution of the Y--B equation to a representation of the braid group, $\B_n$, we require $R$ to be invertible, with many applications to quantum information theory requiring the solutions to be unitary. Hence, while $R_1$ and $R_{-1}$ yield Gaussian solutions, they do not possess the same applications as $R_i$ and $R_{-i}$.  
    \end{rem}

    \begin{prop}
    \label{prop:equivclassslns}
        Let $R_{1},R_{-1},{R_{i}},$ and $R_{-i}$ be defined as above. Then $[R_1, R_{-1}]$ and $[R_i,R_{-i}]$ are distinct equivalence classes of solutions to the Y--B equation in the sense of Definition \ref{def:equiv_sol}.
    \end{prop}

    \begin{proof}
        To see $R_1 \sim
        R_{-1}$, apply Definition \ref{def:equiv_sol} with
        \begin{align*}
            u = 
            \begin{pmatrix}
                1 & 0 \\
                0 & -i
            \end{pmatrix}.
        \end{align*}
    
        Similarly, to see $R_i \sim R_{-i}$, apply Definition \ref{def:equiv_sol} with 
        \begin{align*}
            u = 
            \begin{pmatrix}
            1 & 0 \\
            0 & i
            \end{pmatrix}.
        \end{align*}
        To see that these two equivalence classes are disjoint, observe that $R_1$ is not unitary, while $R_i$ is. 
    \end{proof}
    
    \section{The Faddeev-Reshetikhin-Takhtajan (FRT) construction}
    \label{sec:FRTSect}

In this section, we employ the Fadeev--Reshetikhin--Takhtajan construction \cite{FRT}, commonly known as the FRT construction, to Gaussian solutions $R$, described in Equations (\ref{sln1}) and (\ref{eqn:Ri}),  to obtain bialgebras, which we denote by $A(R)$. We adapt the results detailed in \cite[\S 5.4, Theorem 60]{CMZ} to our setting. We observe that in prior work (by \cite{LR1998} and others) the FRT construction has been used to produce bialgebras, however, our approach involves working specifically with Gaussian solutions and analyzing the resulting families of bi/Hopf algebras. 

We first briefly describe the FRT construction. Let $M$ be the two-dimensional vector space spanned by $\{e_0, e_1\}$ and let $M^{*}$ be the dual module to $M$ generated by the dual basis $\{p^0, p^1\}$, that is, $\langle p^i, e_j \rangle = \delta^i_j$. We then have the free basis $\{c^i_j = e_j \otimes p^i \}$ of $\End_{\Bbbk}(M^*) \cong M \otimes M^*$ with the isomorphism defined by $c^i_j(p^k) = \delta^k_jp^i$ for $0 \le i, j \le 1$. Now, let $C = (C, \Delta, \varepsilon)$ be the comatrix coalgebra and $\rho: M \to M \otimes C$ be defined by $\rho(m_l) = m_v \otimes c^v_l$. The pair $(M, \rho)$ is a right $C$-comodule if and only if the matrix $(c^v_l)$ is comultiplicative, that is, if and only if 
\begin{align}
\label{comultmatrix}
\Delta(c^j_k) &= c^j_u \otimes c^u_k \text{ and } \varepsilon(c^j_k) = \delta^j_k,
\end{align}
for all $i, j, k, l = 0, 1$. Recall that the Einstein summation convention is being used. 

We construct the tensor algebra $T(C)$ by endowing it with a unique bialgebra structure $(T(C), M, u, \overline{\Delta}, \overline{\varepsilon})$ such that $\overline{\Delta}(c) = \Delta(c)$ and $\overline{\varepsilon}(c) = \varepsilon(c)$ for all $c \in C$. Moreover, the linear map $\rho$ endows $M$ with a right $T(C)$--comodule structure. One can then compute the Yang-Baxter obstructions, $z^{ij}_{kl}$, by the formula
\begin{align}
\label{z^ij_kl}
z^{ij}_{kl} = x^{ij}_{vu}c^u_kc^v_l - x^{vu}_{lk}c^i_vc^j_u,
\end{align}
for all $i, j, k, l = 0, 1$. Consider the two-sided ideal of $T(C)$, $I$, generated by $z^{ij}_{kl}$. As $I$ is a biideal of $T(C)$, 
\[A(R) = T(C)/I\]
is a bialgebra and $M$ has a right $A(R)$-comodule structure. Thus, $A(R)$ is the free algebra that is generated by $c^i_j$ for $i, j = 0, 1$ such that the Y--B obstructions, $z^{ij}_{kl}$, vanish for all $i, j, k, l = 0, 1$

Now, consider $R = R_1$, the Gaussian solution corresponding to $\omega = 1$, given in Example \ref{sln1}.
We denote by $x^{ij}_{uv}$ the family of scalars such that $R(e_u \otimes e_v) = x^{ij}_{uv} e_{i} \otimes e_j$ for all $u,v = 0,1$. 

\begin{theorem}
\label{thm:A(R1)_presentation}
The bialgebra $A(R_1)$ is generated by $a$, $b$, $c$, and $d$ and is subject to the following relations
\begin{align*}
a^2 &= d^2, & b^2 &= c^2, & ab &= dc, & ac &= db,\\ 
ad &= da, & ba &= cd, & bc &= cb, & bd &= ca,
\end{align*}
{\noindent where $c^0_0 = a, c^0_1 = b, c^1_0 = c,$ and $c^1_1 = d$. Moreover, the comultiplication and the counit of $A(R_1)$ are given as follows.}
\begin{align*}
\Delta(a) &= a \otimes a + b \otimes c,&\Delta(b) &= a \otimes b + b \otimes d,&\Delta(c) &= c \otimes a + d \otimes c,\\
\Delta(d) &= c \otimes b + d \otimes d, 
&\varepsilon(a) &= 1, \quad \varepsilon(b) = 0, &
&\varepsilon(c) = 0, \quad
\varepsilon(d) = 1. 
\end{align*}
\end{theorem}

\begin{proof}
In Table \ref{tab:R1_zijkl} of Appendix \ref{app:A1zijkl}, we use the formula given in Equation (\ref{z^ij_kl}) to compute the Y--B obstructions, $z^{ij}_{kl}$, for all $i, j, k, l = 0, 1$ and thus, the ideal $I$. The above presentation of $A(R)$ follows immediately.
\end{proof}

\begin{theorem}
\label{thm:A(Ri)_presentation}
The bialgebra $A(R_i)$ is generated by $a$, $b$, $c$, and $d$ and is subject to the following relations
\begin{align*}
a^2 &= d^2, & b^2 &= c^2, & ab &= ba = 0, & ac &= ca = 0,\\  bd &= db = 0 , & cd &= dc = 0, & bc &= cb, & ad &= 2bc + da.
\end{align*}
Moreover, the comultiplication and the counit of $A(R_i)$ are identical to those of the bialgebra, $A(R_1)$. 
\end{theorem}
\begin{proof}
The proof follows similarly to that of Theorem \ref{thm:A(R1)_presentation}. The Y--B obstructions, $z^{ij}_{kl}$, for all $i, j, k, l = 0, 1$, when $R = R_i$ are listed in Table \ref{tab:Ri_zijkl} in Appendix \ref{app:A1zijkl}. 
\end{proof}

\begin{rem}
The bialgebras $A(R)$ described in Theorems \ref{thm:A(R1)_presentation} and \ref{thm:A(Ri)_presentation} above are coquasitriangular. To define the bilinear map $\sigma: A(R) \otimes A(R) \to \Bbbk$ such that $(A(R), \sigma)$ is coquasitritriangular and that the Gaussian solution $R$ remains invariant with respect to $\sigma$, first define $\sigma$ on $C \otimes C$, extend to $T(C)$, and ensure that all hypotheses on $\sigma$ of are satisfied \cite[see e.g., Definition 7, \S5.3]{CMZ}.
\end{rem}

\section{Twisting a Gaussian Solution to the constant quantum Yang-Baxter equation}
\label{sec:twists}
The idea of a twist of a $\mathbb{Z}$-graded algebra was introduced in \cite[Section 8]{ATV1991}. With the notion of twisting systems, Zhang showed in \cite[Theorem 3.1]{Zhang1996} that for two $\mathbb{N}$-graded algebras $A$ and $B$ generated in degree one, $A$ is isomorphic to a twisted algebra of $B$ if and only if the corresponding graded module categories are equivalent.

\begin{defn} 
Let $A$ be a $\mathbb {Z}$-graded algebra and $\phi$ be a graded automorphism of $A$. The \textit{right Zhang twist} $A^\phi$ of $A$ by $\phi$ is the algebra that coincides with $A$ as a graded $\Bbbk$-vector space, with the twisted (or deformed) multiplication 
\[r *_\phi s = r \phi^{|r|}(s),
\,\, \text{for any homogeneous elements } r,s \in A. \]
Here we denote the grading degree of $r$ by $|r|$. Similarly, the \textit{left Zhang twist} $\prescript{\phi}{}{A}$ of $A$ by $\phi$ is defined with the twisted  multiplication 
\[r *_\phi s = \phi^{|s|}(r) s,
\,\, \text{for any homogeneous elements } r,s \in A. \]
We omit the subscript in the twisted multiplication $ *_\phi$ below. 
\end{defn} 

In prior work by the second author \cite{HNUVVW1}, algebras twisted by Zhang twists are associated to 2-cocycle twisted bialgebras via the notion of a twisting pair of algebra automorphisms.

\begin{defn}\cite[Definition D.]{HNUVVW1}
\label{def:twisting_pair}
    Let $(B, m, u, \Delta, \varepsilon)$ be a bialgebra over $\Bbbk$. We say that a pair $(\phi_1, \phi_2)$ of algebra automorphisms of $B$ is a twisting pair if 
    \begin{align*}
        \Delta \circ \phi_1 &= (\id \otimes \phi_1) \circ \Delta, \\
        \Delta \circ \phi_2 &= (\phi_2 \otimes \id) \circ \Delta, \\
        \varepsilon \circ (\phi_1 \circ \phi_2) &= \varepsilon.
    \end{align*}
\end{defn}


To twist a constant (parameter-independent) Gaussian solution $R \in \End_{\Bbbk}(M) \otimes \End_{\Bbbk}(M)$ by a Zhang twist given by an automorphism of $A(R)$, we consider invertible complex $2\times2$ matrices. By work in \cite{HNUVVW1}, we know that such a twist will be defined by a twisting pair. Indeed, in \cite[Lemma 4.1.4]{HNUVVW1}, every twisting pair $(\phi_1, \phi_2)$ of algebra automorphisms of $A(R)$ is determined by its action on the generators $\{a,b,c,d\} = \{c^0_0, c^0_1, c^1_0, c^1_1\} = \{c^i_j\}_{0 \leq i,j \leq 1}$ such that 
\begin{align}
\label{eqn:phi1Andphi2}
    \phi_1(c^i_j) = \sum_{0 \leq u \leq 1} \alpha^i_uc^u_j \quad \text{ and } \quad\phi_2(c^i_j) = \sum_{0 \leq u \leq 1} \beta^u_j c^i_u 
\end{align}
where $\alpha = (\alpha^i_j)$ is an invertible element of $M_2(\C)$ with inverse $\beta =  (\beta^i_j)$, and satisfying 
\begin{align}
\label{eqn:twist_conditions}
\sum_{0 \leq u,v \leq 1} R^{ij}_{vu} \alpha^{u}_{k} \alpha^{v}_{l} = \sum_{0 \leq u,v \leq 1} R^{vu}_{lk}\alpha^i_v \alpha^j_u  
\end{align}
for all $0 \leq i,j,k,l \leq 1.$ In the notation $c^i_j$, $i$ is the column number and $j$ is the row number.

In Theorem \ref{thm:valid_pairs} below, we determine the $2 \times 2$ complex matrices that define a twisting pair.
\begin{theorem}
Suppose $D = (\alpha^i_j) \in M_2(\C)$ is invertible, for $i, j = 0, 1$.
\label{thm:valid_pairs}
\begin{enumerate}
\item If the matrix $D$ yields a twisting pair for the Gaussian solution $R_1 \in \End_{\Bbbk}(M) \otimes \End_{\Bbbk}(M)$, then it takes on one of the following forms
\begin{align*}
    \text{I.} \begin{pmatrix}
            \lambda_1 & \lambda_2 \\
            \lambda_2 & \lambda_1
        \end{pmatrix} \quad
        \text{II.}\begin{pmatrix}
            \lambda_1 & \lambda_2 \\
            -\lambda_2 & -\lambda_1
        \end{pmatrix}
\end{align*}
where $\lambda_1,\lambda_2 \in \C$.
\item If the matrix $D$ yields a twisting pair for the Gaussian solution $R_{i} \in \End_{\Bbbk}(M) \otimes \End_{\Bbbk}(M)$, then it is of the following form 
\label{validmatrixform}
\begin{align*}
        \begin{pmatrix}
            \lambda & 0 \\
            0 & \pm \lambda
        \end{pmatrix},
    \end{align*}
    where $\lambda \in \C$.
\end{enumerate}
\end{theorem}
\begin{proof}
Let 
\[D = \begin{pmatrix}
        \lambda_1 & \lambda_2 \\
        \lambda_3 & \lambda_4
\end{pmatrix} \in M_2(\C)\]
be invertible. It follows that $\lambda_1 \lambda_4 \neq \lambda_2 \lambda_3$.
We consider first the case where $R = R_1$. By Equation \ref{eqn:twist_conditions}, for each tuple $(i,j,k,l) = \{0,1\}$, as we sum over $u$ and $v$, 
\begin{align*}
R^{ij}_{vu} \alpha^{u}_{k} \alpha^{v}_{l} 
&= 
R^{vu}_{lk}\alpha^i_v \alpha^j_u,
\end{align*}
for $u, v = \{0, 1\}$.
This leads to sixteen equations in $\lambda_1$, $\lambda_2$, $\lambda_3$, and $\lambda_4$, provided in Table \ref{tab:R1_twisting_pairs_4.1}, and the matrix given in part (1) follows immediately. In the case where $R = R_i$, Equation \ref{eqn:twist_conditions} leads to sixteen equations described in Table \ref{tab:Ri_twisting_pairs_4.1}. The matrix given in part (2) follows immediately.



\end{proof}

\begin{prop}

\label{prop:pairs_R1}
Retain the notation of Theorem \ref{thm:valid_pairs} above. The twisting pairs $(\phi_1, \phi_2)$ of algebra automorphisms of $A(R_1)$ are defined as follows. 

\begin{table}[h]
    \centering
    \def\arraystretch{1.75}
    \begin{tabular}{l c c}
    \toprule
         & $\phi_1(c^i_j)$ & $\phi_2(c^i_j)$  \\
    \midrule
    Form I & $\lambda_1c^{i}_j + \lambda_2c^{i+1}_{j}$ & $\frac{\lambda_1 c^{i}_j - \lambda_2 c^{i}_{j+1}}{\lambda_1^2 - \lambda_2^2}$\\
    Form II & $(-1)^i(\lambda_1c^{i}_j - \lambda_2 c^{i+1}_{j})$ & $(-1)^{j} \frac{\lambda_1 c^i_j + \lambda_2 c^i_{j+1}}{\lambda_1^2 - \lambda_2^2}$\\
    \bottomrule
    \end{tabular}
    
    \caption{The twisting pairs $(\phi_1, \phi_2)$ corresponding to $A(R_1)$.}
    \label{tab:twisting_pairs_R1}
\end{table}
\end{prop}

\begin{proof}   
This follows from Appendix \ref{app:A1zijkl}, and Equation \ref{eqn:phi1Andphi2} above. Note that all indices are read modulo 2.
\end{proof}

\begin{prop}
\label{prop:pairs_Ri}
Retain the notation of Theorem \ref{thm:valid_pairs} above. The twisting pairs $(\phi_1, \phi_2)$ of algebra automorphisms of $A(R_i)$ are defined as follows
\end{prop}
\begin{table}[H]
    \centering
    \def\arraystretch{1.75}
    \begin{tabular}{l c c}
    \toprule
    $\phi_1(c^i_j)$ & $\phi_2(c^i_j)$  \\
    \midrule
    $\pm\lambda c^{i}_j$  & $(\pm \lambda)^{-1} c^i_j$\\
    \bottomrule
    \end{tabular}
    \caption{The twisting pairs $(\phi_1, \phi_2)$ corresponding to $A(R_i)$.}
    \label{tab:twisting_pairs_Ri}
\end{table}

\begin{proof}
    This follows from Appendix \ref{app:A1zijkl} and Equation \ref{eqn:phi1Andphi2} above.
\end{proof}

In the next result, we twist the solutions $R_1$ and $R_i$ using the twisting pairs $(\phi_1, \phi_2)$ given in Propositions \ref{prop:pairs_R1} and \ref{prop:pairs_Ri}. A 2-cocycle twist is first defined from $(\phi_1, \phi_2)$ and applied to $R_1$ and to $R_i$ to produce twisted solutions.

\begin{defn}[2-cocycle]
A \textit{left 2-cocycle} on a bialgebra $B$ is a convolution invertible linear map $\sigma: B \otimes B \to \Bbbk$ satisfying
\label{2cocycle}  
\[
    \sum \sigma(x_1, y_1)\, \sigma (x_2 y_2, z) = \sum \sigma(y_1,z_1)\, \sigma(x,y_2z_2), 
\]
for all $x,y,z \in B$. Similarly, a \textit{right 2-cocycle} on $B$ is a convolution invertible bilinear map $\sigma: B \times B \rightarrow \Bbbk$ satisfying 
\[
    \sum \sigma(x_2, y_2)\, \sigma (x_1 y_1, z) = \sum \sigma(y_2,z_2)\, \sigma(x,y_1z_1).
\]
For any 2-cocycle $\sigma$, we denote its convolution inverse by $\sigma^{-1}$.
\end{defn}

\begin{theorem}
\label{thm:twisted_R1}
Suppose that the matrices $\alpha$ and $\beta = \alpha^{-1} \in M_2(\mathbb{C})$ are given by twisting pairs $(\phi_1, \phi_2)$ of algebra automorphisms of $A(R_1)$ and that $2 \times 2$ matrices of Forms I and II are as defined in Theorem \ref{thm:valid_pairs}(1). If the Gaussian solution $R_{1} \in \End_{\Bbbk}(\C^2) \otimes \End_{\Bbbk}(\C^2)$ is twisted by the 2-cocycle $\sigma: A(R_1) \times A(R_1) \to \Bbbk$, then
the twisted solution, $(R_1)^{\sigma}$, is a non-Gaussian solution when $\alpha \in M_2(\C)$ is of Form I with diagonal entries summing to zero, and the twisted solution $(R_1)^{\sigma}$ is a non-Gaussian solution when $\alpha \in M_2(\C)$ is of Form II.
\end{theorem}

\begin{proof}
As a result of Theorem \ref{thm:valid_pairs}, the Gaussian solution $R_1$ can be twisted by matrices  $(\alpha^i_j)$ of Forms I or II,
\begin{align*}
I = \begin{pmatrix}
            \lambda_1 & \lambda_2 \\
            \lambda_2 & \lambda_1
        \end{pmatrix} & \text{ or }
II = \begin{pmatrix}
            \lambda_1 & \lambda_2 \\
            -\lambda_2 & -\lambda_1
        \end{pmatrix},
\end{align*}
to yield twisting pairs.
Consider first the case where the matrix $(\alpha^i_j)$ is of form I. Using the twisting pair $(\phi_1, \phi_2)$ of $A(R)$ given in Proposition \ref{prop:pairs_R1}, and by work of the second author \cite[Proposition 2.3.2]{HNUVVW1}, we may define a 2-cocycle twist, $\sigma:A(R) \otimes A(R) \to \Bbbk$ by
\begin{align}
\label{2cocycleformula}
\sigma(x,y) &= \varepsilon(x)\varepsilon((\phi_2)^{|x|}(y)),
\end{align}
for any homogeneous elements $x, y \in A(R)$, where $|x|$ denotes the degree of $x$. Applying Theorem \ref{thm:A(R1)_presentation} with respect to generators $a = c^0_0, b = c^0_1, c = c^1_0$, and $d = c^1_1$ we obtain,
\begin{align*}
\sigma(a,a) &= \frac{\lambda_1}{\lambda_1^2 - \lambda_2^2} = \sigma(d, d),
\end{align*}
and $\sigma(x, y) = 0$ for all $x$ and $y$ in $A(R)\backslash \{\Bbbk a + \Bbbk d\}$. The convolution inverse $\sigma^{-1}$ may be computed similarly using the formula $\sigma^{-1}(x,y) = \varepsilon(x)\varepsilon((\phi_1)^{|y|}(y))$, for all homogeneous elements $x, y \in A(R)$.
Note that the twisted solution $R^{\sigma}$ may be computed by applying $\sigma$ to $R_1$ or by using Equation \ref{eqtn:twist_R} \cite[Corollary 4.2.5]{HNUVVW1} given below.
\begin{align}
\label{eqtn:twist_R}
    \left( R^{\sigma} \right)^{ij}_{kl} = \sum_{0 \leq p,v \leq 1} \alpha^i_p R^{pj}_{kv}\beta^v_l \quad \text{ for all } 0 \leq i,j,k,l \leq 1.
\end{align}

When the matrix $\alpha = (\alpha^i_j) \in M_2(\mathbb{C})$ is of Form I, we split this into three subcases: (a) $\lambda_1 = 0$, (b) $\lambda_2 = 0$, and (c) $\lambda_1\lambda_2 \neq 0$.

\begin{enumerate}[label=(\emph{\alph*})]
\item When $\lambda_1 = 0$, the twisted solution
\begin{align*}
(R_1)^{\sigma} =\begin{pmatrix}
    1 & 0 & 0 & -1\\
    0 & -1 & 1 & 0\\
    0 & 1 & -1 & 0 \\
    -1 & 0 & 0 & 1
    \end{pmatrix} \ne R_1.
\end{align*}    

\item When $\lambda_2 = 0$, the twisted solution $(R_1)^{\sigma} = R_1$.

\item When $\lambda_1\lambda_2 \ne 0$, the twisted solution $(R_1)^{\sigma} = R_1$
\end{enumerate}
Observe that the twisted solution $(R_1)^{\sigma}$ is not invertible, as expected. Moreover, $(R_1)^{\sigma}$ yields a Gaussian solution only when $\lambda_1 = 0$.

Consider now the case where the matrix $\alpha = (\alpha^i_j) \in M_2(\mathbb{C})$ is of Form II. As in the previous case, we consider three subcases: (a) $\lambda_1 = 0$, (b) $\lambda_2 = 0$, and (c) $\lambda_1\lambda_2 \ne 0$. Using the twisting pair $(\phi_1, \phi_2)$ of $A(R)$ given in Proposition \ref{prop:pairs_R1},
\begin{enumerate}[label=(\emph{\alph*})]
\item When $\lambda_1 = 0$,
\begin{align*}
(R_1)^{\sigma} =\begin{pmatrix}
    1 & 0 & 0 & -1\\
    0 & -1 & 1 & 0\\
    0 & 1 & -1 & 0 \\
    -1 & 0 & 0 & 1
    \end{pmatrix} \ne R_1.
\end{align*}    

\item When $\lambda_2 = 0$, 
\begin{align*}
(R_1)^{\sigma} 
= \begin{pmatrix}
    1 & 0 & 0 & -1\\
    0 & -1 & 1 & 0\\
    0 & 1 & -1 & 0 \\
    -1 & 0 & 0 & 1
    \end{pmatrix}
\ne R_1
\end{align*}
\item When $\lambda_1\lambda_2 \ne 0$, \begin{align*}
(R_1)^{\sigma} = 
        \begin{pmatrix}
        1 & 0 & 0 & -1\\
        0 & -1 & 1 & 0\\
        0 & 1 & -1 & 0 \\
        -1 & 0 & 0 & 1
        \end{pmatrix} \ne R_1.
\end{align*} 
\end{enumerate}
\end{proof}
In this case, $(R_1)^{\sigma} \ne R_1$ and is not invertible either, as expected. Moreover, $(R_1)^{\sigma}$ is not a Gaussian solution as it is not equivalent to any of the solutions given in Example \ref{egR1}.   

\begin{theorem}
\label{thm:twisted_Ri}
Suppose that the matrices $\alpha$ and $\beta = \alpha^{-1} \in M_2(\mathbb{C})$ are given by twisting pairs $(\phi_1, \phi_2)$ of algebra automorphisms of $A(R_i)$ defined by matrices in Theorem \ref{thm:valid_pairs}(2). If the Gaussian solution $R_{i} \in \End_{\Bbbk}(\C^2) \otimes \End_{\Bbbk}(\C^2)$ is twisted by the 2-cocycle $\sigma: A(R_i) \times A(R_i) \to \Bbbk$, then
the twisted solution, $(R_i)^{\sigma}$, is a Gaussian solution, when $\alpha \in M_2(\C)$ has a single eigenvalue, and the twisted solution $(R_i)^{\sigma}$ is a non-Gaussian solution when $\alpha \in M_2(\C)$ has two distinct eigenvalues.
\end{theorem}

\begin{proof}
We consider the following two forms of matrices from Theorem \ref{thm:valid_pairs}
\begin{align*}
I = \begin{pmatrix}
\lambda & 0\\
0 & \lambda
\end{pmatrix}, \qquad
II = \begin{pmatrix}
\lambda & 0\\
0 & -\lambda
\end{pmatrix}.
\end{align*}
We use Equation \ref{eqtn:twist_R} to obtain the twisted solution $(R_i)^{\sigma}$

In the case where $\alpha$ corresponds to Form I:
\begin{align*}
    R_i^{\sigma} = 
        \begin{pmatrix}
            1 & 0 & 0 & i \\
            0 & 1 & -1 & 0 \\
            0 & 1 & 1 & 0 \\
            i & 0 & 0 & 1
        \end{pmatrix},
    \end{align*}
which equals $R_i$.

In the case where $\alpha$ corresponds to Form II,
we obtain by Equation (\ref{eqtn:twist_R})
\begin{align*}
(R_i)^{\sigma} &= 
\begin{pmatrix}
1 & 0 & 0 & -i \\
0 & -1 & -1 & 0 \\
0 & 1 & -1 & 0 \\
-i & 0 & 0 & 1
\end{pmatrix}.
\end{align*}
We observe that $(R_i)^{\sigma}$ is not equivalent to any of the Gaussian solutions given in Example \ref{egR1}.    
\end{proof}

\bibliographystyle{amsalpha}

\pagebreak

\appendix
\markboth{Appendix}{Appendix}
{\section{Y--B obstructions and Equation \ref{eqn:twist_conditions}} conditions}\label{app:A1zijkl}
\begin{appendices}

\begin{table}[H]
        \centering
        \caption{The 16 Y--B Obstructions corresponding to $R_1$.}
        \label{tab:R1_zijkl}
        \renewcommand{\arraystretch}{1.25}
        \begin{tabular}{c c c c c l}
        \toprule
         & $(u,v)$ & & & &  \\
        $(i,j,k,l)$ & $(0,0)$ & $(0,1)$ & $(1,0)$ & $(1,1)$ & $z^{ij}_{kl} = \sum$\\
        \midrule
        $(0,0,0,0)$ & $0$ & $0$ & $0$ & $c^2-b^2$ & $z^{00}_{00} = c^2 - b^2$ \\
        $(0,1,0,0)$ & $-ac$ & $ac$ & $ca$ & $-bd$ & $z^{01}_{00} = ca - bd$ \\
        $(0,0,0,1)$ & $ab$ & $-ba$ & $-ab$ & $cd$ & $z^{00}_{01} = cd - ba$\\
        $(0,1,0,1)$ & $0$ & $ad-bc$ & $cb-ad$ & $0$ & $z^{01}_{01} = cb - bc$ \\
        $(0,0,1,0)$ & $ba$ & $-ba$ & $-ab$ & $dc$ & $z^{00}_{10} = dc - ab$\\
        $(0,1,1,0)$ & $0$ & $0$ & $da - ad$ & $0$ & $z^{01}_{10} = da - ad$\\
        $(0,0,1,1)$ & $b^2-a^2$ & $0$ & $0$ & $d^2-b^2$ & $z^{00}_{11} = d^2 - a^2$ \\
        $(0,1,1,1)$ & $-ac$ & $bd$ & $db$ & $-bd$ & $z^{01}_{11} = db - ac$ \\
        $(1,0,0,0)$ & $-ca$ & $ac$ & $ca$ & $-db$ & $z^{10}_{00} = ac - db$ \\
        $(1,0,0,1)$ & $0$ & $ad - da$ & $0$ & $0$ & $z^{10}_{01} = ad - da$ \\
        $(1,0,1,0)$ & $0$ & $bc - da$ & $da - cb$ & $0$ & $z^{10}_{10} = bc - cb$ \\
        $(1,0,1,1)$ & $-ca$ & $bd$ & $db$ & $-db$ & $z^{10}_{11} = bd - ca$ \\
        $(1,1,0,0)$ & $a^2 - c^2$ & $0$ & $0$ & $c^2 - d^2$ & $z^{11}_{00} = a^2 - d^2$ \\
        $(1,1,0,1)$ & $ab$ & $-dc$ & $-cd$ & $cd$ & $z^{11}_{01} = ab - dc$ \\
        $(1,1,1,0)$ & $ba$ & $-dc$ & $-cd$ & $dc$ & $z^{11}_{10} = ba - cd$ \\
        $(1,1,1,1)$ & $b^2 - c^2$ & $0$ & $0$ & $0$ & $z^{11}_{11} = b^2 - c^2$ \\
        \bottomrule
        \end{tabular}
        \label{tab:R1_obstructions}
    \end{table}

\pagebreak

\begin{table}[H]
        \centering
        \caption{The 16 Y--B Obstructions corresponding to $R_i$.}
        \label{tab:Ri_zijkl}
        \renewcommand{\arraystretch}{1.25}
        \begin{tabular}{c c c c c l}
        \toprule
         & $(u,v)$ & & & &  \\
        $(i,j,k,l)$ & $(0,0)$ & $(0,1)$ & $(1,0)$ & $(1,1)$ & $z^{ij}_{kl} = \sum$\\
        \midrule
        $(0,0,0,0)$ & $0$ & $0$ & $0$ & $-i^2(b^2 - c^2)$ & $z^{00}_{00} = -i(b^2 - c^2)$ \\
        $(0,1,0,0)$ & $-ac$ & $ac$ & $ca$ & $-ibd$ & $z^{01}_{00} = ca-ibd$ \\
        $(0,0,0,1)$ & $ab$ & $-ba$ & $-ab$ & $icd$ & $z^{00}_{01} = icd - ba$ \\
        $(0,1,0,1)$ & $0$ & $ad - bc$ & $-ad + cb$ & $0$ & $z^{01}_{01} = cb - bc$ \\
        $(0,0,1,0)$ & $ba$ & $ba$ & $-ab$ & $idc$ & $z^{00}_{10} = 2ba - ab + idc$ \\
        $(0,1,1,0)$ & $0$ & $2bc$ & $-ad + da$ & $0$ & $z^{01}_{10} = 2bc + da - ad$ \\
        $(0,0,1,1)$ & $-ia^2 + b^2$ & $0$ & $0$ & $-b^2 + id^2$ & $z^{00}_{11} = -i^2(a^2 - d^2)$ \\
        $(0,1,1,1)$ & $-iac$ & $bd$ & $db$ & $-bd$ & $z^{01}_{11} = db - iac$ \\
        $(1,0,0,0)$ & $-ca$ & $ac$ & $-ca$ & $-idb$ & $z^{10}_{00} = ac - 2ca - idb$ \\
        $(1,0,0,1)$ & $0$ & $ad - da$ & $-2cb$ & $0$ & $z^{10}_{01} = ad - 2cb - da$ \\
        $(1,0,1,0)$ & $0$ & $bc + da$ & $-cb - da$ & $0$ & $z^{10}_{10} = bc - cb$ \\
        $(1,0,1,1)$ & $-ica$ & $bd$ & $-db$ & $-db$ & $z^{10}_{11} = bd - 2db - ica$ \\
        $(1,1,0,0)$ & $-c^2 + ia^2$ & $0$ & $0$ & $c^2 - id^2$ & $z^{11}_{00} = i(a^2 - d^2)$ \\
        $(1,1,0,1)$ & $iab$ & $-dc$ & $-cd$ & $cd$ & $z^{11}_{01} = iab - dc$ \\
        $(1,1,1,0)$ & $iba$ & $dc$ & $-cd$ & $dc$ & $z^{11}_{10} = 2dc + iba - cd$ \\
        $(1,1,1,1)$ & $-ic^2 + ib^2$ & $0$ & $0$ & $0$ & $z^{11}_{11} = i(b^2 - c^2)$ \\
        \bottomrule
        \end{tabular}
    \end{table}

\pagebreak

\begin{table}[H]
        \centering
        \caption{The 16 conditions from Equation \ref{eqn:twist_conditions} corresponding to $R_1$}
        \label{tab:R1_twisting_pairs_4.1}
        \renewcommand{\arraystretch}{1.25}
        \begin{tabular}{c c c c c}
        \toprule 
        $(i,j,k,l)$ & Left Sum & Right Sum & Difference \\
        \midrule 
        $(0,0,0,0)$ & $\lambda_1^2 + \lambda_2^2$ & $\lambda_1^2 + \lambda_3^2$ & $\lambda^2_2 - \lambda_3^2$ \\
        $(0,0,0,1)$ & $\lambda_1\lambda_3 + \lambda_2 \lambda_4$ & $2 \lambda_1 \lambda_3$  & $-\lambda_1\lambda_3 + \lambda_2\lambda_4$ \\
        $(0,0,1,0)$ & $\lambda_1 \lambda_3 + \lambda_2 \lambda_4$ & $2 \lambda_1 \lambda_3$  & $-\lambda_1\lambda_3 + \lambda_2\lambda_4$ \\
        $(0,0,1,1)$ & $\lambda_3^2 + \lambda_4^2$ & $\lambda_1^2 + \lambda_3^2$ & $-\lambda_1^2 + \lambda_4^2$\\
        $(0,1,0,0)$ & $2 \lambda_1 \lambda_2$ & $\lambda_1 \lambda_2 + \lambda_3 \lambda_4$  & $\lambda_1 \lambda_2 - \lambda_3 \lambda_4$\\
        $(0,1,0,1)$ & $\lambda_2 \lambda_3 + \lambda_1 \lambda_4$ & $\lambda_2 \lambda_3 + \lambda_1 \lambda_4$ & $0$ \\
        $(0,1,1,0)$ & $\lambda_2 \lambda_3 + \lambda_1 \lambda_4$ & $\lambda_2 \lambda_3 + \lambda_1 \lambda_4$ & $0$ \\
        $(0,1,1,1)$ & $2\lambda_3\lambda_4$ & $\lambda_1\lambda_2 + \lambda_3\lambda_4$ & $-\lambda_1\lambda_2 + \lambda_3\lambda_4$\\
        $(1,0,0,0)$ & $2 \lambda_1 \lambda_2$ & $\lambda_1 \lambda_2 + \lambda_3 \lambda_4$  & $\lambda_1 \lambda_2 - \lambda_3 \lambda_4$\\
        $(1,0,0,1)$ & $\lambda_2 \lambda_3 + \lambda_1 \lambda_4$ & $\lambda_2 \lambda_3 + \lambda_1 \lambda_4$ & $0$ \\
        $(1,0,1,0)$ & $\lambda_2 \lambda_3 + \lambda_1 \lambda_4$ & $\lambda_2 \lambda_3 + \lambda_1 \lambda_4$ & $0$ \\
        $(1,0,1,1)$ & $2\lambda_3\lambda_4$ & $\lambda_1\lambda_2 + \lambda_3\lambda_4$ & $-\lambda_1\lambda_2 + \lambda_3\lambda_4$ \\
        $(1,1,0,0)$ & $\lambda_1^2 + \lambda_2^2$ & $\lambda_2^2 + \lambda_4^2$ & $\lambda_1^2 - \lambda_4^2$ \\
        $(1,1,0,1)$ & $\lambda_1 \lambda_3 + \lambda_2 \lambda_4$ & $2 \lambda_2 \lambda_4$  & $\lambda_1 \lambda_3 - \lambda_2 \lambda_4$ \\
        $(1,1,1,0)$ & $\lambda_1 \lambda_3 + \lambda_2 \lambda_4$ & $2 \lambda_2 \lambda_4$  & $\lambda_1 \lambda_3 - \lambda_2 \lambda_4$ \\
        $(1,1,1,1)$ & $\lambda_3^2 + \lambda_4^2$ & $\lambda_2^2 + \lambda_4^2$ & $-\lambda_2^2 + \lambda_3^2$ \\
        \bottomrule
        \end{tabular}
\end{table}

\pagebreak

\begin{table}[H]
        \centering
        \caption{The 16 conditions from Equation \ref{eqn:twist_conditions} corresponding to $R_i$}
        \label{tab:Ri_twisting_pairs_4.1}
        \renewcommand{\arraystretch}{1.25}
        \begin{tabular}{c c c c c}
        \toprule 
        $(i,j,k,l)$ & Left Sum & Right Sum & Difference \\
        \midrule 
        $(0,0,0,0)$ & $\lambda_1^2 + i \lambda_2^2$ & $\lambda_1^2 + i \lambda_3^2$ & $i(\lambda_2^2 - \lambda_3^2)$ \\
        $(0,0,0,1)$ & $\lambda_1\lambda_3 + i \lambda_2 \lambda_4$ & $2 \lambda_1 \lambda_3$ & $-\lambda_1\lambda_3 + i \lambda_2 \lambda_4$\\
        $(0,0,1,0)$ & $\lambda_1 \lambda_3 + i \lambda_2 \lambda_4$ & $0$ & $\lambda_1 \lambda_3 + i \lambda_2 \lambda_4$\\
        $(0,0,1,1)$ & $\lambda_3^2 + i \lambda_4^2$ & $i \lambda_1^2 + \lambda_3^2$ & $-i(\lambda_1^2 - \lambda_4^2)$\\
        $(0,1,0,0)$ & $2 \lambda_1 \lambda_2$ & $\lambda_1 \lambda_2 + i \lambda_3 \lambda_4$ & $\lambda_1 \lambda_2 - i \lambda_3 \lambda_4$\\
        $(0,1,0,1)$ & $\lambda_2 \lambda_3 + \lambda_1 \lambda_4$ & $\lambda_2 \lambda_3 + \lambda_1 \lambda_4$  & $0$\\
        $(0,1,1,0)$ & $\lambda_2 \lambda_3 + \lambda_1 \lambda_4$ & $-\lambda_2 \lambda_3 + \lambda_1 \lambda_4$ & $2 \lambda_2 \lambda_3$\\
        $(0,1,1,1)$ & $2 \lambda_3 \lambda_4$ & $i \lambda_1 \lambda_2 + \lambda_3 \lambda_4$ & $-i\lambda_1\lambda_2 + \lambda_3 \lambda_4$\\
        $(1,0,0,0)$ & $0$ & $\lambda_1 \lambda_2 + i \lambda_3 \lambda_4$ & $-\lambda_1 \lambda_2 - i \lambda_3 \lambda_4$\\
        $(1,0,0,1)$ & $-\lambda_2 \lambda_3 + \lambda_1 \lambda_4$ & $\lambda_2 \lambda_3 + \lambda_1 \lambda_4$ & $-2\lambda_2 \lambda_3$\\
        $(1,0,1,0)$ & $\lambda_2 \lambda_3 - \lambda_1 \lambda_4$ & $\lambda_2 \lambda_3 - \lambda_1 \lambda_4$ & $0$\\
        $(1,0,1,1)$ & $0$ & $i \lambda_1 \lambda_2 + \lambda_3 \lambda_4$ & $-i \lambda_1 \lambda_2 - \lambda_3 \lambda_4$\\
        $(1,1,0,0)$ & $i\lambda_1^2 + \lambda_2^2$ & $\lambda_2^2 + i \lambda_4^2$ & $i(\lambda_1^2 - \lambda_4^2)$ \\
        $(1,1,0,1)$ & $i \lambda_1 \lambda_3 + \lambda_2 \lambda_4$ & $2 \lambda_2 \lambda_4$ & $i \lambda_1 \lambda_3 - \lambda_2 \lambda_4$\\
        $(1,1,1,0)$ & $i \lambda_1 \lambda_3 + \lambda_2 \lambda_4$ & $0$ & $i \lambda_1 \lambda_3 + \lambda_2 \lambda_4$ \\
        $(1,1,1,1)$ & $i \lambda_3^2 + \lambda_4^2$ & $i \lambda_2^2 + \lambda_4^2$ & $-i(\lambda_2^2 - \lambda_3^2)$\\
        \bottomrule
        \end{tabular}
\end{table}
\end{appendices}
\vfill
\end{document}